\documentclass[12pt,a4paper]{article}
\usepackage[utf8]{inputenc}
\usepackage{amsmath}
\usepackage{amsfonts}
\usepackage{amssymb}
\usepackage{enumerate}
\usepackage{amsthm}
\usepackage{graphicx}
\usepackage{setspace}

\author{Thomas Perrett\\
Department of Applied Mathematics and Computer Science,\\
Technical University of Denmark,\\ DK-2800 Lyngby, Denmark\\
\texttt{tper@dtu.dk}}
\title{A zero-free interval for chromatic polynomials of graphs with $3$-leaf spanning trees}

\newtheorem{theorem}{Theorem}[section]

\newtheorem{lemma}{Lemma}[section]

\newtheorem{proposition}{Proposition}[section]
\newtheorem{conjecture}{Conjecture}[section]

\newtheorem{definition}{Definition}[section]

\begin{document}
\maketitle
\begin{abstract}
It is proved that if $G$ is a graph containing a spanning tree with at most three leaves, then the chromatic polynomial of $G$ has no roots in the interval $(1,t_1]$, where $t_1 \approx 1.2904$ is the smallest real root of the polynomial $(t-2)^6 +4(t-1)^2(t-2)^3 -(t-1)^4$. We also construct a family of graphs containing such spanning trees with chromatic roots converging to $t_1$ from above. We employ the Whitney $2$-switch operation to manage the analysis of an infinite class of chromatic polynomials.
\end{abstract}
\begin{section}{Introduction}

The \emph{chromatic polynomial} $P(G,t)$ of a graph $G$ is a polynomial with integer coefficients which counts for each non-negative integer $t$, the number of $t$-colourings of $G$. It was introduced by Birkhoff~\cite{chrompolyIntroduced} in 1912 for planar graphs, and extended to all graphs by Whitney~\cite{Whitney2,Whitney1} in 1932. If $t$ is a real number then we say $t$ is a \emph{chromatic root} of $G$ if $P(G,t)= 0$. Thus the numbers $0, 1, 2, \dots, \chi (G)-1$ are always chromatic roots of $G$ and, in fact, the only rational ones. On the other hand it is easy to see that $P(G,t)$ is never zero for $t\in (-\infty,0)$ and Tutte~\cite{Tutte} showed that the same is true for the interval $(0,1)$. In 1993 Jackson~\cite{32/27Jackson} proved the surprising result that the interval $(1, 32/27]$ is also zero-free, and found a sequence of graphs whose chromatic roots converge to $32/27$ from above. Thomassen~\cite{RootsDenseThomassen} strengthened this by showing that the set of chromatic roots consists of $0$, $1$, and a dense subset of the interval $(32/27, \infty)$.

Let $Q(G,t) = (-1)^{|V(G)|}P(G,t)$, and $b(G)$ be the number of blocks of $G$. We say that $G$ is \emph{separable} if $b(G) \geq 2$ and \emph{non-separable} otherwise. Note that $K_2$ is non-separable. In~\cite{HamPathThomassen}, Thomassen provided a new link between Hamiltonian paths and colourings by proving that the zero-free interval of Jackson can be extended when $G$ has a Hamiltonian path. More precisely he proved the following.

\begin{theorem}\emph{\cite{HamPathThomassen}}\label{thm:HamPathThomassen}
If $G$ is a non-separable graph with a Hamiltonian path, then $Q(G,t)>0$ for $t \in (1,t_0]$, where $t_0 \approx 1.295$ is the unique real root of the polynomial $(t-2)^3+4(t-1)^2$. Furthermore for all $\varepsilon >0$ there exists a non-separable graph with a Hamiltonian path whose chromatic polynomial has a root in the interval $(t_0, t_0+\varepsilon)$.
\end{theorem}

If $G$ is separable and has a Hamiltonian path then it is easily seen using Theorem~\ref{thm:HamPathThomassen} and Proposition~\ref{prop:factorcomplete} that $Q(G,t)$ is non-zero in the interval $(1,t_0]$ with sign $(-1)^{b(G)-1}$.

For a graph $G$, a \emph{$k$-leaf spanning tree} is a spanning tree of $G$ with at most $k$ leaves (vertices of degree 1). We denote the class of non-separable graphs which admit a $k$-leaf spanning tree by $\mathcal{G}_k$. Thus Theorem~\ref{thm:HamPathThomassen} gives a zero-free interval for the class $\mathcal{G}_2$. In this article we prove the following analogous result for the class $\mathcal{G}_3$.

\begin{theorem}\label{thm:mainThm}
Let $G$ be a non-separable graph with a $3$-leaf spanning tree, then $Q(G,t)>0$ for $t \in (1,t_1]$, where $t_1 \approx 1.2904$ is the smallest real root of the polynomial $(t-2)^6 +4(t-1)^2(t-2)^3 -(t-1)^4$. Furthermore for all $\varepsilon >0$ there exists a non-separable graph with a $3$-leaf spanning tree whose chromatic polynomial has a root in the interval $(t_1, t_1+\varepsilon)$.
\end{theorem}

A natural extension of this work would be to find $\varepsilon_k >0$ so that $(1,32/27 + \varepsilon_k]$ is zero-free for the class $\mathcal{G}_k$, $k\geq 4$. However since the graphs presented by Jackson~\cite{32/27Jackson} are non-separable, it must be that $\varepsilon_k \rightarrow 0$ as $k \rightarrow \infty$. Another possible extension would be to find $\varepsilon_l >0$ so that $(1, 32/27+\varepsilon_l]$ is zero-free for the family of graphs containing a spanning tree $T$ with $\Delta(T) \leq 3$ and at most $l$ vertices of degree $3$. Here the possible implications are much more interesting since it is not clear if $\varepsilon_l \rightarrow 0$ as $l \rightarrow \infty$. Indeed only finitely many of the graphs in~\cite{32/27Jackson} have a spanning tree of maximum degree $3$. Theorem~\ref{thm:HamPathThomassen} and our result solve the cases $l=0$ and $l=1$ respectively, which leads us to conjecture the following.

\begin{conjecture}\label{conj:}
There exists $\varepsilon>0$ such that if $G$ is a non-separable graph with a spanning tree of maximum degree $3$, then $Q(G,t)>0$ for $t \in (1, 32/27 + \varepsilon]$.  
\end{conjecture} 

Barnette~\cite{Barnette} proved that a $3$-connected planar graph has a spanning tree of maximum degree $3$. Thus an affirmative answer to Conjecture~\ref{conj:} would immediately imply a zero-free interval for the class of $3$-connected planar graphs. Such an interval is known~\cite{Nearly3Conn} but thought to be far from best possible. 

\end{section}

\begin{section}{Preliminaries}

All graphs in this article are \emph{simple}, that is they have no loops or multiple edges. If $uv$ is an edge of $G$ then $G/uv$ denotes the graph obtained by deleting $uv$ and then identifying $u$ with $v$. This operation is referred to as the \emph{contraction} of $uv$. If $G$ is connected, $S \subset V(G)$ and $G-S$ is disconnected, then $S$ is called a \emph{cut-set} of $G$. A \emph{$2$-cut} of $G$ is a cut-set $S$ with $|S|=2$. If $S$ is a cut-set of $G$ and $C$ is a component of $G-S$, then we say the graph $G[V(C) \cup S]$ is an \emph{$S$-bridge} of $G$. If $P$ is a path and $x,y \in V(P)$ then $P[x,y]$ denotes the subpath of $P$ from $x$ to $y$.

We make repeated use of two fundamental results in the study of chromatic polynomials.

\begin{proposition}[Deletion-contraction identity]\label{prop:delcont}
Let $G$ be a graph and $uv$ be an edge of $G$. Then $$P(G,t) = P(G-uv,t) - P(G/uv,t).$$
\end{proposition}

\begin{proposition}[Factoring over complete subgraphs]\label{prop:factorcomplete}
Let $G = G_1 \cup G_2$ be a graph such that $G[V(G_1) \cap V(G_2)]$ is a complete graph on $k$ vertices. Then $$P(G,t) = \frac{P(G_1,t)P(G_2,t)}{P(K_{k},t)}.$$
\end{proposition}
The next proposition is easily proven using Propositions~\ref{prop:delcont} and~\ref{prop:factorcomplete}. The operation involved is often called a \emph{Whitney $2$-switch}.
\begin{proposition}\label{prop:WhitneySwitch}
Let $G$ be a graph and $\{x,y\}$ be a $2$-cut of $G$. Let $C$ denote a connected component of $G - \{x,y\}$. Define $G'$ to be the graph obtained from the disjoint union of $G-C$ and $C$ by adding for all $z \in V(C)$ the edge $xz$ (respectively $yz$) if and only if $yz$ (respectively $xz$) is an edge of $G$. Then we have $P(G,t) = P(G',t).$
\end{proposition}

If $G'$ can be obtained from $G$ by a sequence of Whitney $2$-switches, then $P(G,t) = P(G',t)$ and we say $G$ and $G'$ are \emph{Whitney equivalent}.

\begin{definition}
Let $G$ be a graph. We say $G$ has property $\Delta$ if the following conditions hold:
\begin{list}{-}{}
\item $G$ is non-separable but not 3-connected.
\item For every $2$-cut $\{x,y\}$, $xy \notin E(G)$ and $G$ has precisely three $\{x,y\}$-bridges, all of which are separable.
\end{list}
\end{definition}

The property $\Delta$ characterises the class of \emph{generalised triangles} defined by Jackson in~\cite{32/27Jackson}. Such graphs can be formed from a triangle by repeatedly replacing an edge by two paths of length $2$.

\begin{subsection}{Hamiltonian Paths}\label{subsec:Hampaths}
For each natural number $k\geq 1$, let $H_k$ denote the graph obtained from a path $x_1 x_2 \dots x_{2k+3}$ by adding the edges $x_1 x_4$, $x_{2k}x_{2k+3}$, and all edges $x_{i}x_{i+4}$ for $i \in \{2, 4, 6, \dots, 2k-2\}$. Also define $H_0 = K_3$ and let $\mathcal{H} = \{H_i : i \in \mathbb{N}_0\}$. In~\cite{HamPathThomassen}, Thomassen showed that if $H$ is a smallest counterexample to Theorem~\ref{thm:HamPathThomassen}, then $H$ is isomorphic to $H_k$ for some $k \in \mathbb{N}_0$. Since $t_0$ is taken as the infimum of the non-trivial chromatic roots of all $H_i \in \mathcal{H}$, a contradiction follows. 

For fixed $t \in (1, t_0]$ the value of the chromatic polynomial of $H_k$ at $t$ can be expressed as 
\begin{equation}\label{eq:H_k}
P(H_{k},t) = A \alpha^k + B \beta^k
\end{equation}
where $A, B, \alpha$ and $\beta$ are constants depending on $t$, defined by the following relations~\cite{HamPathThomassen}.
\begin{eqnarray}
\delta = \sqrt{(t-2)^4 +4(t-1)^2 (t-2)}\label{eqn:delta}\\
\alpha = \frac{1}{2}((t-2)^2 + \delta ), \:\:\: \beta = \frac{1}{2}((t-2)^2 - \delta )\label{eqn:alphabeta}\\
A+B = t(t-1)(t-2)\label{eqn:A+B}\\
A\alpha + B\beta = t(t-1)((t-2)^3 + (t-1)^2)
\end{eqnarray}
Noting that $\alpha + \beta = (t-2)^2$, we have more explicitly
$$A = \frac{1}{\delta}t(t-1)((t-2)\alpha + (t-1)^2)$$
$$B = t(t-1)(t-2)-A.$$
As stated in~\cite{HamPathThomassen}, it can be seen that for $t \in (1,t_0)$, $0<\beta<\alpha<1$ and $0<B<-A<1$. The graphs in $\mathcal{H}$ and the quantities defined above will play an important role in our result.

\end{subsection}

\begin{section}{A Special Class of Graphs with $3$-leaf Spanning Trees}\label{sec:SpecialClass}

Let $F_k = H_k -x_1x_2$. If $G$ is a graph, $\{x,y\}$ is a $2$-cut of $G$ and $B$ is an $\{x,y\}$-bridge, then we write $B = F(x,y,k)$ to indicate that $B$ is isomorphic to $F_k$, where $x$ is identified with $x_1$ and $y$ is identified with $x_2$ in $G$. For $i,j,k \in \mathbb{N}_0$, define $G_{i,j,k}$ to be the graph composed of two vertices $x$ and $y$, and three $\{x,y\}$-bridges $F(x,y,i)$, $F(x,y,j)$ and $F(y,x,k)$. Figure~\ref{fig:G423} shows the graph $G_{4,2,3}$. Note that if $j=0$, then $G_{i,j,k}$ is isomorphic to $H_{i+k+2}$.

\begin{figure}
\centering
\includegraphics[scale=1]{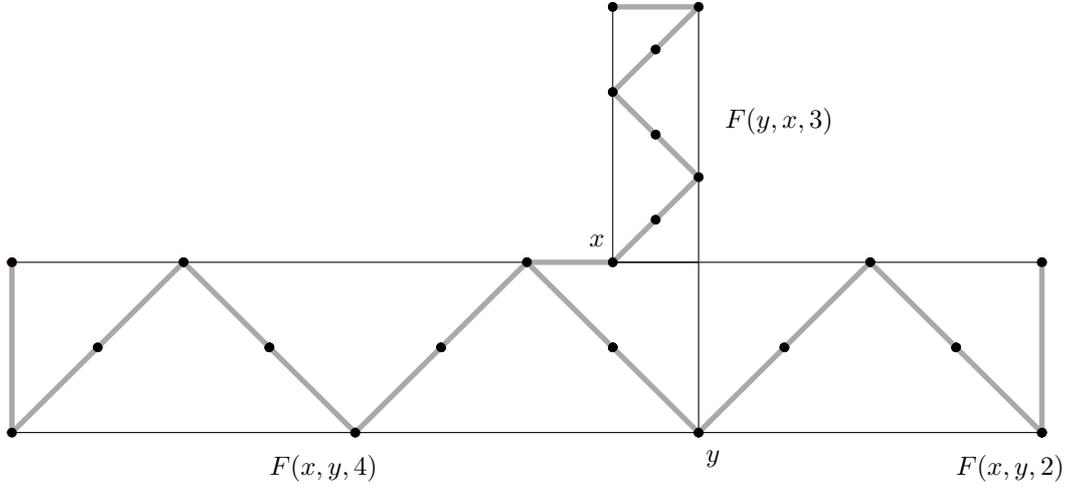}
\caption{The graph $G_{4,2,3}$ and a $3$-leaf spanning tree thereof.}\label{fig:G423}
\end{figure}

Lemma~\ref{lem:BridgeCharacThomassen} characterises particular bridges of a graph satisfying property $\Delta$. It can be found in the work of Thomassen~\cite{HamPathThomassen} and will be useful for us as a lemma.

\begin{lemma}\label{lem:BridgeCharacThomassen}
Let $G$ be a graph with property $\Delta$, $\{x,y\}$ be a $2$-cut of $G$ and $B$ be an $\{x,y\}$-bridge of $G$.

\begin{enumerate}[(a)]
\item If $B$ contains a Hamiltonian path $P$ starting at $x$ and ending at $y$, then $B = F(x,y,0)$, i.e. $B$ is a path of length 2.
\item If $B$ contains a path $P$ starting at $y$ and covering all vertices of $B$ except for $x$, then $B = F(x,y,k)$ for some $k \in \mathbb{N}_0$.
\end{enumerate} 
\end{lemma}

\begin{proof}
\begin{enumerate}[(a)]
\item Since $G$ has property $\Delta$, $B$ is separable and has a cut-vertex $v$. The Hamiltonian path $P$ of $B$ shows that neither of $G-\{x,v\}$ and $G-\{y,v\}$ can have more than two components. Thus since $G$ has property $\Delta$, neither $\{x,v\}$ nor $\{y,v\}$ is a cut-set of $G$ and so $|V(B)| = 3$. Since $B$ is connected and $xy \not\in E(G)$, $B$ is a path of length $2$ as claimed. 
\item Again $B$ is separable and has a cut-vertex $v$. If $|V(B)|=3$ then $B = F(x,y,0)$, so we may assume that $|V(B)|\geq 4$ and the result is true for all bridges on fewer vertices. At least one of $\{x,v\}$ or $\{y,v\}$ is a $2$-cut of $G$, but $\{x,v\}$ cannot be since $G-\{x,v\}$ has at most two components. Thus $xv\in E(G)$ and $\{y,v\}$ is a $2$-cut of $G$ with precisely three $\{y,v\}$ bridges, two of which, say $B_1$ and $B_2$, are contained in $B$. Suppose without loss of generality that $B_1$ contains the subpath of $P$ from $x$ to $v$. Then $P[V(B_1)]$ is a Hamiltonian path of $B_1$ and so by part~(a), $B_1 = F(y,v,0)$. $P[V(B_2)]$ is a path in $B_2$, starting at $v$ and covering all vertices of $B_2$ except for $y$. By induction, $B_2 = F(y,v,k-1)$ for some $k-1 \in \mathbb{N}_0$. It follows that $B = F(x,y,k)$.
\end{enumerate}
\end{proof}

It is easy to see that each $G_{i,j,k}$ has property $\Delta$ and contains a $3$-leaf spanning tree. The following result shows that it is enough to only consider these graphs.

\begin{lemma}\label{lem:charac}
Let $G$ be a graph with a $3$-leaf spanning tree $T$. If $G$ satisfies property $\Delta$, then $P(G,t) = P(G_{i,j,k},t)$ for some $i,j,k \in \mathbb{N}_0$.
\end{lemma}

\begin{proof}
We show that $G$ is Whitney equivalent to $G_{i,j,k}$ for some $i,j,k \in \mathbb{N}_0$. By the remark following Proposition~\ref{prop:WhitneySwitch}, this implies the result. If $G$ contains a Hamiltonian path then for any $2$-cut $\{x,y\}$, one of the three $\{x,y\}$-bridges satisfies Lemma~\ref{lem:BridgeCharacThomassen}(a), whilst the other two satisfy Lemma~\ref{lem:BridgeCharacThomassen}(b). Thus $G$ is Whitney equivalent to $G_{i,0,k}$ for some $i,k \in \mathbb{N}_0$ and we are done. 

Now we may assume that $v$ is a vertex of degree $3$ in $T$. We first find a useful $2$-cut. Since $G$ is not $3$-connected, there is some cut set of size $2$. Choose such a $2$-cut $S = \{x,y\}$ so that the smallest $S$-bridge containing $v$ is as small as possible. We claim that $v \in S$. If this is not the case then let $B$ be the $S$-bridge containing $v$. Since $B$ is separable there is some cut-vertex $u$ of $B$. Also, since $v$ has degree $3$ in $T$, $|V(B)| \geq 4$. This implies that one of $\{x,u\}$ or $\{y,u\}$ is a smaller cut-set containing $v$, a contradiction. We now claim to be able to find a $2$-cut $S$ such that $v \in S$, and the three neighbours of $v$ in $T$, denoted $v_1, v_2$ and $v_3$, lie in three different $S$-bridges. If this is not already the case, choose a $2$-cut $S$ such that $v \in S$, and the $S$-bridge containing two of $v_1, v_2, v_3$ is as small as possible. By a similar argument we find a $2$-cut with the desired property.

Fix the $2$-cut $S = \{x,y\}$ so that $y$ has degree $3$ in $T$. For $i \in \{1,2,3\}$, let $B_i$ be an $S$-bridge, and $y_i \in V(B_i)$ be the neighbours of $y$ in $T$. Finally for $i \in \{1,2,3\}$ we let $P_i$ be the unique path in $T$ from $y$ to a leaf of $T$, which contains the vertex $y_i$. Suppose without loss of generality that $x$ lies on $P_2$. We distinguish two cases.

\begin{figure}
\centering
\includegraphics[scale=0.8]{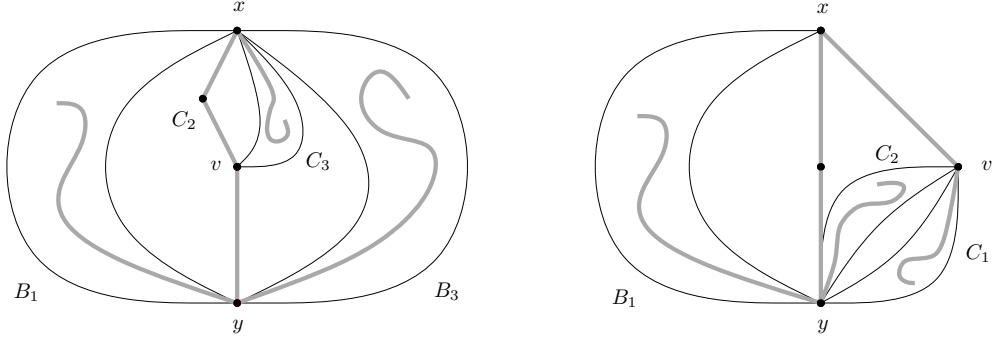}
\caption{Cases 1 and 2 in the proof of Lemma~\ref{lem:charac}.}\label{fig:charac}
\end{figure}

\textbf{Case 1: $V(P_2) = V(B_2)$.}\\
$P_1$ and $P_3$ are paths in $B_1$ and $B_3$ respectively, which start at $y$ and cover all vertices of $B_1, B_3$ except $x$. By Lemma~\ref{lem:BridgeCharacThomassen}(b), $B_1 = F(x,y,i)$ and $B_3 = F(x,y,k)$ for some $i,k \in \mathbb{N}_0$. If $P_2$ ends at $x$ then Lemma~\ref{lem:BridgeCharacThomassen}(a) implies $B_2 = F(x,y,0)$ and by performing a Whitney $2$-switch of $B_3$ about $\{x,y\}$ we are done. So suppose $P_2$ ends at some vertex $z$ other than $x$. $B_2$ is separable and so there is a cut-vertex $v$ of $B_2$. Since $P$ contains a subpath connecting $y$ and $x$, it follows that $v$ lies between $y$ and $x$ on $P$. So let $Q_1$, $Q_2$ and $Q_3$ be the sub-paths of $P_2$ from $y$ to $v$, $v$ to $x$ and $x$ to $z$ respectively. Now $G-\{y,v\}$ contains at most two components, and therefore $Q_1$ is the edge $yv$. Note that $|V(B)| \geq 4$, so $G-\{x,v\}$ contains at least two components. It follows from property $\Delta$ that $G$ has precisely three $\{x,v\}$-bridges, two of which are contained in $B_2$. Let $C_2$ and $C_3$ be the $\{x,v\}$-bridges containing $Q_2$ and $Q_3$ respectively. $Q_2$ is a Hamiltonian path of $C_2$ from $v$ to $x$. Thus by Lemma~\ref{lem:BridgeCharacThomassen}(a), $C_2 = F(x,v,0)$. Similarly $C_3$ contains a path starting at $x$ and covering all vertices of $C_3$ except for $v$. By Lemma~\ref{lem:BridgeCharacThomassen}(b) it follows that $B_3 = F(v,x, j-1)$ for some $j-1 \in \mathbb{N}_0$. Now performing a Whitney $2$-switch of $C_3$ with respect to $\{x,v\}$ gives a new graph, where the $\{x,y\}$-bridge corresponding to $B_2$ is $F(x,y,j)$. Finally, performing a Whitney $2$-switch of $B_3$ about $\{x,y\}$ gives a graph isomorphic to $G_{i,j,k}$. This completes the proof.

\textbf{Case 2: $V(P_2) \supset V(B_2)$.}\\
Suppose without loss of generality that $P_2$ also contains vertices of $B_3-x-y$. As before, Lemma~\ref{lem:BridgeCharacThomassen}(b) implies that $B_1 = F(x,y,i-1)$ for some $i-1 \in \mathbb{N}_0$ and $B_2 = F(x,y,0)$ by Lemma~\ref{lem:BridgeCharacThomassen}(a). Now $T[V(B_3)]$ consists of two disjoint paths $Q_1$ and $Q_2$, starting at $x$ and $y$ respectively. Since $B_3$ is separable, there is a cut-vertex $v$. Suppose without loss of generality that $v \in V(Q_1)$. If this is not the case we perform a Whitney $2$-switch of $B_3$ with respect to $\{x,y\}$ and proceed similarly. Both $Q_1$ and $Q_2$ contain at least one edge, thus $|V(B_3)| \geq 4$ and at least one of $\{x,v\}$ and $\{y,v\}$ is a $2$-cut of $G$. Suppose for a contradiction that $\{x,v\}$ is a $2$-cut. $G - \{x,v\}$ has at least two components and so $G$ has precisely three $\{x,v\}$-bridges, two of which are contained in $B_3$. Now $G- \{v,y\}$ can have at most two components and as such $v$ is the unique neighbour of $y$ in $B_3$. But $v \in V(Q_1)$ and therefore $Q_2$ is a single vertex. This contradicts the fact that $Q_2$ contains at least one edge.

So we may assume that $\{y,v\}$ is a $2$-cut of $G$, and $v$ is the unique neighbour of $x$ in $B_3$. Now, as before, $G$ has precisely three $\{v,y\}$-bridges, two of which are contained in $B_3$. Denote the two $\{v,y\}$-bridges which are contained in $B_3$ by $C_1$ and $C_2$, where $Q_2$ is contained in $C_2$. It is now easy to see that $Q_2$ is a path of $C_2$ starting at $y$ and containing all vertices of $C_2$ except for $v$. Similarly, $Q_1[V(C_1)]$ is a path of $C_1$ starting at $v$ and containing all vertices of $C_1$ except for $y$.  Thus by Lemma~\ref{lem:BridgeCharacThomassen}(b), $C_1 = F(y,v,k)$ and $C_2 = F(v,y,j)$ for some $j,k \in \mathbb{N}_0$.

Now consider the $2$-cut $\{v,y\}$ of $G$. It has three bridges two of which are $C_1$ and $C_2$ found in $B_3$. The third $\{v,y\}$-bridge, denoted $C_3$, is composed of the edge $xv$ and the two $\{x,y\}$-bridges, $F(x,y,0)$ and $F(x,y,i-1)$ of $G$. By performing a Whitney $2$-switch of $F(x,y,i-1)$ at $\{x,y\}$, we get a new graph $G'$, where the $\{v,y\}$-bridge of $G'$ corresponding to $C_3$ is precisely $F(v,y,i)$. Now $G'$ is isomorphic to $G_{i,j,k}$ and so $P(G,t) = P(G_{i,j,k},t)$. This completes the proof.
\end{proof}

\end{section}
\begin{subsection}{A Zero Free Interval for $P(G_{i,j,k},t)$}

We now determine the behaviour of the chromatic roots of each $G_{i,j,k}$. Here $t_0$ is the real number defined in Theorem~\ref{thm:HamPathThomassen} and $\mathcal{H} = \{H_i : i \in \mathbb{N}_0\}$ is the family of graphs defined in Section~\ref{subsec:Hampaths}. It is easily seen that $G_{i,0,k} = H_{i+k+2}$. If $j=1$ then by adding and contracting the edge from $x$ to the cut vertex of $F(x,y,j)$ we find that $$P(G_{i,1,k},t) = (t-2)^2 P(H_{i+k+2},t) + \frac{(t-1)}{t} P(H_{i+1},t)P(H_{k},t).$$ Finally if $j \geq 2$ then using Proposition~\ref{prop:delcont} and~\ref{prop:factorcomplete} gives the recurrence $$P(G_{i,j,k},t) = (t-2)^2 P(G_{i,j-1,k},t) + (t-1)^2 (t-2) P(G_{i,j-2,k},t).$$

Solving this explicitly for fixed $t \in (1, t_0]$ gives a solution of the form $P(G_{i,j,k},t) = C \alpha^j + D \beta^j$ where $C$ and $D$ are constants depending on $i,k$ and $t$. Recall $\alpha$ and $\beta$ are defined in~(\ref{eqn:alphabeta}). The initial conditions corresponding to $j=0,1$ are
\begin{equation}\label{eqn:initj=0}
C+D = P(H_{i+k+2},t)
\end{equation}
\begin{equation}\label{eqn:initj=1}
C\alpha + D\beta = (t-2)^2 P(H_{i+k+2},t) + \frac{t-1}{t} P(H_{i+1},t) P(H_{k},t).
\end{equation}
Multiplying~(\ref{eqn:initj=0}) by $\beta$ and subtracting the resulting equation from (\ref{eqn:initj=1}) gives 
\begin{align}
C(\alpha - \beta) &= ((t-2)^2 - \beta)P(H_{i+k+2},t) + \frac{t-1}{t}P(H_{i+1},t)P(H_{k},t)\notag \\
&= \alpha P(H_{i+k+2},t) + \frac{t-1}{t}P(H_{i+1},t)P(H_{k},t) \label{eqn:init-init}.
\end{align}
For convenience we define $\gamma = \gamma(t) = \alpha t /(t-1) >0,$ for $t \in (1,t_0]$. Let $t_1$ be the smallest real root of the polynomial $(t-2)^6 +4(t-1)^2(t-2)^3 -(t-1)^4$. We claim that for $t \in (1, t_1]$
\begin{equation}\label{eqn:gamma}
 \gamma\beta < -A \leq \gamma \alpha.
\end{equation}

The left inequality follows since $\gamma \beta  = -t(t-1)(t-2)$ and so by~(\ref{eqn:A+B}), $-A-\gamma\beta = B>0.$ The right side follows since for $t \in (1,t_0]$
\vspace{6pt}

\begin{tabular}{lrcl}
 & $-A$ & $\leq$ & $\gamma\alpha$ \\ 
$\iff$ &  $-\frac{1}{\delta}t(t-1)((t-2)\alpha+(t-1)^2)$ & $\leq$ & $\frac{t}{t-1}(t-2)((t-2)\alpha +(t-1)^2)$ \\ 
$\iff$ & $-(t-1)^2$ & $\geq$ & $(t-2)\delta$\\ 
$\iff$ & $ (t-1)^4$ & $\leq$ & $(t-2)^2 \delta^2.$
\end{tabular}
\vspace{6pt}

\noindent Using~(\ref{eqn:delta}), the final inequality is seen to be satisfied when the aforementioned polynomial is non-negative.

Since each $H \in \mathcal{H}$ has an odd number of vertices, Theorem~\ref{thm:HamPathThomassen} implies $P(H,t)<0$ for $t \in (1,t_0]$. It now follows that (\ref{eqn:init-init}), and hence $C$, are negative if $$\gamma|P(H_{i+j+2},t)|>P(H_{i+1},t)P(H_{k},t).$$
Indeed for $t \in (1,t_1]$, we have that $0<\beta < \alpha <1$ and $0<B<-A<1$, which together with~(\ref{eq:H_k}) and~(\ref{eqn:gamma}) implies
\begin{align*}
P(H_{i+1},t)P(H_{k},t) &= (A\alpha^{i+1}+B\beta^{i+1})(A\alpha^{k}+B\beta^{k})\\
&= A^2 \alpha^{i+k+1} + AB\alpha^{i+1}\beta^{k} + AB\alpha^{k}\beta^{i+1} + B^2\beta^{i+k+1}\\
&< -A\gamma\alpha^{i+k+2} - B^2\beta^{i+k+1} - B\gamma\beta^{i+k+2} + B^2\beta^{i+k+1}\\
&= \gamma(-A\alpha^{i+k+2} - B\beta^{i+k+2}) = \gamma|P(H_{i+k+2},t)|.
\end{align*}
Since $C<0$,~(\ref{eqn:initj=0}) implies $D<-C$. Finally, since $\alpha>\beta$, we may conclude that $P(G_{i,j,k},t) <0$ for $t \in (1,t_1]$.

Now suppose that $t \in (t_1, t_0)$ is fixed. Then $-A > \gamma \alpha$. Setting $i+1 = k$ for simplicity we see that $$\frac{P(H_{k},t)^2}{\gamma |P(H_{2k+1},t)|} = \frac{A^2 \alpha^{2k}+2AB\alpha^k \beta^k +B^2 \beta^{2k}}{\gamma (-A \alpha^{2k+1} -B\beta^{2k+1})} \longrightarrow \frac{A^2}{-\gamma A \alpha} = \frac{-A}{\gamma \alpha} >1.$$ as $k \rightarrow \infty$. Thus, for large enough $i$ and $k$, $\gamma|P(H_{i+k+2},t)|< P(H_{i+1},t)P(H_{k},t)$ and hence $C$ is positive. Though~(\ref{eqn:initj=0}) implies that $D$ is negative, since $\alpha > \beta$ it follows that for large enough $j$, $P(G_{i,j,k},t) >0$. Since we have proven that $P(G_{i,j,k},t) <0$ on $(1,t_1]$, we may conclude by continuity that $P(G_{i,j,k},t)$ has a root in $(t_1, t)$. 
\end{subsection}
\end{section}

\begin{section}{Main Result}

To prove Theorem~\ref{thm:mainThm} we shall show that a smallest counterexample has property~$\Delta$. In~\cite{DongKoh}, Dong and Koh extracted the essence of the proofs of Thomassen~\cite{HamPathThomassen} and Jackson~\cite{32/27Jackson} and gave a general method to do this. An important part of that method is the following definition and lemma.

\begin{definition}
Let $\mathcal{G}$ be a family of graphs. $\mathcal{G}$ is called splitting-closed if the following conditions hold for each $G \in \mathcal{G}$.
\begin{list}{-}{}
\item For every complete cut-set $C$ with $|C|\leq 2$, all $C$-bridges of $G$ are in $\mathcal{G}$.
\item For every $2$-cut $\{x,y\}$ such that $xy \not\in E(G)$, and every $\{x,y\}$-bridge $B$, the graphs $B+xy$ and $B/xy$ are in $\mathcal{G}$.
\end{list} 
\end{definition}

It is straightforward to check that the family of non-separable graphs with a $3$-leaf spanning tree is splitting-closed. In fact $\mathcal{G}_k$ is splitting-closed for all $k\geq 2$.

\begin{lemma}(Adapted from~\emph{\cite{DongKoh}})\label{lem:dongKoh}
Let $G$ be a non-separable graph and $\{x,y\}$ be a $2$-cut of $G$ with $\{x,y\}$-bridges $B_1, \dots, B_m$ where $m$ is odd. For fixed $i,j \in [m]$, let $B_{i,j}$ be the graph formed from $B_i \cup B_j$ by adding a new vertex $w$ and the edges $xw$ and $wy$. Let $B_{\cup} = \cup_{k \in [m]\setminus \{i,j\}}B_k$. For fixed $t \in (1,2)$, if $Q(B_{\cup},t)$, $Q(B_{i,j},t)$, $Q(B_{k}+xy,t)$, $Q(B_{k}/xy,t)$ are positive for each $k \in [m]\setminus \{i,j\}$, then $Q(G,t)>0$. 
\end{lemma}

The proof of Lemma~\ref{lem:dongKoh} is a straightforward modification of the proof of Lemma~2.5 in~\cite{DongKoh}. We now prove the main theorem.

\begin{proof}[Proof of Theorem~\ref{thm:mainThm}]
Let $G$ be a smallest counterexample to the theorem and let $t \in (1, t_1]$ such that $Q(G,t)\leq 0$. We show that $G$ has property $\Delta$ and hence by Lemma~\ref{lem:charac}, $P(G,t) = P(G_{i,j,k},t)$ for some $i,j,k \in \mathbb{N}_0$. Since no $G_{i,j,k}$ has a root less than or equal to $t_1$, a contradiction ensues. 

By the hypotheses, $G$ is non-separable.

\textbf{Claim 1: $G$ is not $3$-connected.}

If $G$ is $3$-connected then $G-e$ and $G/e$ are non-separable for every edge $e \in E(G)$. So let $v$ be a leaf of $T$, and $e$ be an edge incident to $v$ but not in $T$. Then also $G-e$ and $G/e$ have a $3$-leaf spanning tree. By Proposition~\ref{prop:delcont} it follows that $G$ is not a smallest counterexample. This is a contradiction.

In what follows let $\{x,y\}$ be an arbitrary $2$-cut with $\{x,y\}$-bridges $B_1, \dots, B_m$.

\textbf{Claim 2: For every $2$-cut $\{x,y\}$, $xy \not\in E(G)$.}

Suppose $xy \in E(G)$. By Proposition~\ref{prop:factorcomplete}, $$Q(G,t) = \frac{Q(B_1,t)Q(B_2,t) \cdots Q(B_m,t)}{t^{m-1}(t-1)^{m-1}}.$$  Since $G$ is a smallest counterexample and $\mathcal{G}_3$ is splitting-closed, $Q(B_i,t) >0$ for $i \in [m]$. A contradiction follows.

\textbf{Claim 3: All $2$-cuts have precisely three bridges.}

Suppose $m$ is even, then 
\begin{align*}
Q(G,t) &= Q(G+xy,t) - Q(G/xy,t)\\
&= \frac{Q(B_1+xy,t) \cdots Q(B_m+xy,t)}{t^{m-1}(t-1)^{m-1}} + \frac{Q(B_1/xy,t) \cdots Q(B_m/xy,t)}{t^{m-1}}.
\end{align*}

Since $G$ is a smallest counterexample and $\mathcal{G}_3$ is splitting-closed, all terms in the final expression are positive and so $Q(G,t)>0$, a contradiction. Thus $m$ is odd. If $m\geq 5$, choose two bridges $B_i$ and $B_j$ for which $T[V(B_i)]$ and $T[V(B_j)]$ are disconnected and form the graphs $B_{i,j}$ and $B_{\cup}$ as described in Lemma~\ref{lem:dongKoh}. Clearly $B_{i,j}$ and $B_{\cup}$ are non-separable and have a $3$-leaf spanning tree. The same is true for all $B_{k}+xy$, $B_k /xy$, $k \in [m]\setminus \{i,j\}$ since $\mathcal{G}_3$ is splitting-closed. Since each of these graphs is smaller than $G$, the conditions of Lemma~\ref{lem:dongKoh} are satisfied and thus $Q(G,t)>0$, a contradiction.

\textbf{Claim 4: If $\{x,y\}$ is a $2$-cut, then every $\{x,y\}$-bridge is separable.}

Let $B$ be an arbitrary $\{x,y\}$-bridge, say $B = B_1$, and suppose for a contradiction that $B$ is non-separable. Since $xy\not\in E(G)$, $|V(B)|\geq 4$. We may assume that $B$ contains at most two leaves of $T$, since if $T$ has three leaves in $B$, then $G - \{x,y\}$ has at most two components. Relabelling $x$ and $y$ if necessary there are four cases how $T[V(B)]$ may behave:

\begin{enumerate}[{Case} 1:] 
\item $T[V(B)]$ is connected.
\item $T[V(B)]$ consists of an isolated vertex $x$ and a path starting at $y$ and covering all vertices of $B-x$.
\item $T[V(B)]$ consists of an isolated vertex $x$ and a tree with precisely three leaves, one of which is $y$, covering all vertices of $B-x$.
\item $T[V(B)]$ consists of two disjoint paths starting at $x$ and $y$ respectively and covering all vertices of $B$.
\end{enumerate}

In Case 1, $T[V(B)]$ is a $3$-leaf spanning tree of $B$. In Case 2 adding any edge of $B$ incident with $x$ to $T$ also shows that $B$ contains such a spanning tree. Thus in these two cases, $B, B+xy$, and $B/xy$ are members of $\mathcal{G}_3$. As $G$ is a smallest counterexample we may apply Lemma~\ref{lem:dongKoh} with $i=2, j=3$ which gives $Q(G,t)>0$, a contradiction.
\begin{figure}
\centering
\includegraphics[scale=0.7]{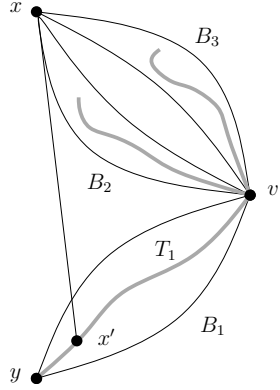}
\caption{Case 3a of Claim 4.}\label{fig:3a}
\end{figure}

Two cases remain.

\textbf{Case 3: $T[V(B)]$ consists of an isolated vertex $x$ and a tree with precisely three leaves, one of which is $y$, covering all vertices of $B-x$.}\\
Let $v$ be the vertex of degree $3$ in $T$ and $T_1$ be the path in $T$ from $y$ to $v$.

\textbf{Subcase 3a: $B - x$ is separable.}\\
Since $|V(B)|\geq 4$, there is a cut-vertex $z$ of $B-x$. So $\{x,z\}$ is a $2$-cut of $B$ and a $2$-cut of $G$. By Claim 3, $G $ has precisely three $\{x,z\}$-bridges, two of which are contained in $B$. This implies that $v = z$ and $T_1$ is a Hamiltonian path of the unique $\{z\}$-bridge of $B-x$ which contains $y$ (see Figure~\ref{fig:3a}). Since $B$ is non-separable, $V(T_1)\geq 3$ and $x$ has a neighbour in $V(T_1)\setminus \{y,z\}$. Choose such a neighbour, $x'$, from which the distance to $y$ on $T_1$ is minimal. It is easy to see that $G$ contains two paths from $x$ to $x'$ avoiding the edge $xx'$ itself. Thus $G - xx'$ is non-separable and has a $3$-leaf spanning tree. By Claim 2, $G/xx'$ is also non-separable. Now Proposition~\ref{prop:delcont} gives 

\begin{equation}\label{eqn:reduction}
Q(G,t) = Q(G - xx',t) + Q(G/xx',t).
\end{equation}

Since $G$ is a smallest counterexample, $Q(G-xx',t)>0$. Thus we have reached a contradiction if $Q(G/xx',t) >0$. This follows immediately if $G/xx'$ has a $3$-leaf spanning tree. Otherwise $|V(B/xx')| \geq 4$, and so we apply Lemma~\ref{lem:dongKoh} to $G/xx'$ with $i=2$ and $j=3$. To see that the hypotheses hold we show that all of the graphs $B/xx'$, $B/xx' + xy$ and $B/xx'/xy$ are non-separable and have a $3$-leaf spanning tree. That they contain such a spanning tree is clear. Also $B/xx' + xy$ is non-separable since $G/xx'$ is non-separable. Finally, if $B/xx'$ or $B/xx'/xy$ is separable, then $B-\{x,x'\}$ or $B-\{x,x',y\}$ is disconnected. By the choice of $x'$, this implies $\{x',y\}$ is a $2$-cut of $G$ with two $\{x',y\}$-bridges, a contradiction to Claim 3.

\textbf{Subcase 3b: $B-x$ is non-separable.}\\
Since $|V(B)|\geq 4$, $G-e$ is non-separable for every edge $e \in E(B)$ incident to $x$. Choose a neighbour, $x'$ of $x$, such that the distance on $T[V(B)]$ from $v$ to $x'$ is maximal. Since $B$ is non-separable, $x$ has at least two neighbours in $B$ and so $x'\neq v$. By Claim 2 and the above, both $G-xx'$ and $G/xx'$ are non-separable. Furthermore $G-xx'$ has a $3$-leaf spanning tree. As before, by~(\ref{eqn:reduction}), we have reached a contradiction if $Q(G/xx') >0$. This follows immediately if $G/xx'$ has a $3$-leaf spanning tree. Otherwise we apply Lemma~\ref{lem:dongKoh} to $G/xx'$ as in Case 3a. The same argument shows that the hypotheses hold.

\textbf{Case 4: Suppose that $T[V(B)]$ consists of two paths $P_1$ and $P_2$, starting at $x$ and $y$ respectively, and covering all vertices of $B$.}\\
Let $P_1 = x_1,\dots,x_{n_1}$ and $P_2 = y_1, \dots, y_{n_2}$ where $x = x_1$ and $y = y_1$. If $x$ or $x_{n_1}$ has a neighbour $x'$ on $P_2$, then $B$ contains a $3$-leaf spanning tree. As in Cases 1 and 2, this is enough to reach a contradiction. Thus $|V(P_1)| \geq 4$ and all neighbours of $x_{n_1}$ lie on $P_1$. Apart from its predecessor on $P_1$, $x_{n_1}$ has at least one other neighbour since $B$ is non-separable. If $x_{n_1}$ has at least two other neighbours, say $x_i$, $x_j$ with $i<j$, then $G-x_{n_1}x_i$ and $G/x_{n_1}x_i$ are non-separable and have a $3$-leaf spanning tree. By Proposition~\ref{prop:delcont} we again reach a contradiction. So we may suppose that $d(x_{n_1})=2$ and $N(x_{n_1}) = \{x_{n_1 - 1}, x_i\}$. It follows that $\{x_{i}, x_{n_1 - 1}\}$ is a $2$-cut of $G$. Thus $G$ has precisely three $\{x_{i}, x_{n_1 - 1}\}$-bridges, of which one contains the subpath $P_1[x_{i},x_{n_1 -1}]$. Call this bridge $B_4$. Note also that there is some edge $e$ from $x_{n_1 -1}$ to the $\{x_{i}, x_{n_1 - 1}\}$-bridge of $G$ containing $y$. Since $T[V(B_4)]$ is a Hamiltonian path of $B_4$, it follows from Case 1 that $B_4$ is separable. Thus it has a cut vertex $v$. Because of the edges $e$ and $x_{i}x_{n_1}$, we see that $G-\{x_i,v\}$ and $G-\{x_{n_1 -1},v\}$ both have at most two components. From Claim 3 it follows that neither of $\{x_i,v\}$ and $\{x_{n_1 -1},v\}$ are cut-sets of $G$. Thus $B_4$ is a path of length $2$, $i = n_1 - 3$, and $d(x_{n_1 -2}) = 2$. We conclude that there is at least one vertex of degree $2$ on the interior of $P_1$. 

\begin{figure}
\centering
\includegraphics[scale=1]{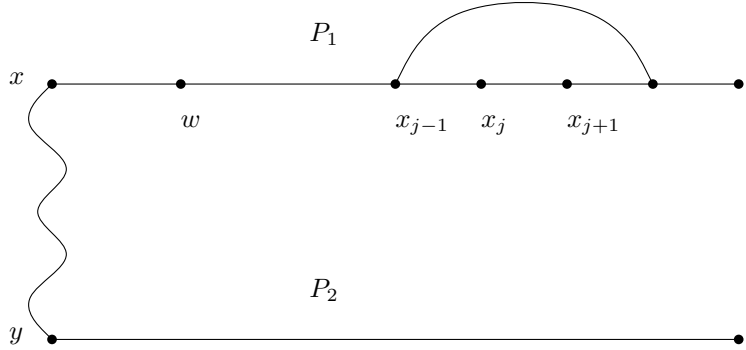}
\caption{If $w \in V(P_1)$.}\label{fig:cutvtxonP1}
\end{figure}

Now let $x_j \in V(P_1)\setminus \{x_1, x_{n_1}\}$ be a vertex of degree $2$ with $j$ as small as possible. Then $\{x_{j-1},x_{j+1}\}$ is a $2$-cut and $G$ contains precisely three $\{x_{j-1},x_{j+1}\}$-bridges. Since each of $x_{j-1}$ and $x_{j+1}$ has a neighbour in each $\{x_{j-1},x_{j+1}\}$-bridge, there is some edge $e$ from $x_{j-1}$ to one of $x_{j+2},x_{j+3}, \dots, x_{n_1}$. Now consider the $\{x_{j-1},x_{j+1}\}$-bridge, $B_y$, containing $y$. This bridge contains a $3$-leaf spanning tree covering all of its vertices apart from $x_{j+1}$. By Case 3, it is separable and has a cut-vertex $w$. It is easy to see that $w$ must lie on $P_1[x,x_{j-2}]$ or $P_2$.
Suppose the former (see Figure~\ref{fig:cutvtxonP1}). Since $|V(B_y)| \geq 4$ at least one of $\{x_{j-1}, w\}$ or $\{x_{j+1},w\}$ is a $2$-cut of $G$. However, because of the edge $e$ and the presence of a path from $x$ to $y$ in another $\{x,y\}$-bridge (indicated in Figure~\ref{fig:cutvtxonP1}), $G - \{x_{j+1},w\}$ has at most two components. Therefore $\{x_{j-1}, w\}$ is a $2$-cut of $G$ and as such gives rise to three $\{x_{j-1}, w\}$-bridges, one of which contains the path $P_1[w,x_{j-1}]$. This path has length at least $2$ by Claim~2. By the argument above, $w = x_{j-3}$ and $d(x_{j-2})=2$, contradicting the minimality of $j$.

\begin{figure}
\centering
\includegraphics[scale=1]{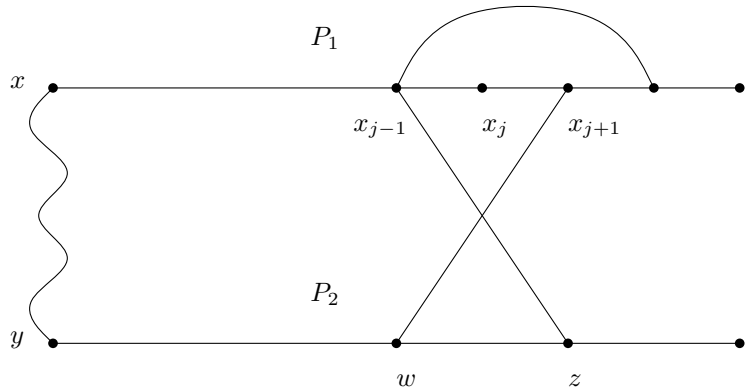}
\caption{If $w \in V(P_2)$.}\label{fig:cutvtxonP2}
\end{figure}
So $w$ lies on $P_2$ (see Figure~\ref{fig:cutvtxonP2}). Once again at least one of $\{x_{j-1}, w\}$ or $\{x_{j+1},w\}$ is a $2$-cut of $G$, but $\{x_{j+1},w\}$ cannot be since $G - \{x_{j+1}, w\}$ has at most two components. It follows that $wx_{j+1}$ is an edge and $\{x_{j-1}, w\}$ is a $2$-cut with precisely three bridges. Let $B_5$ be the $\{x_{j-1}, w\}$-bridge containing the vertex $y_{n_2}$, then $P_2[w,y_{n_2}]$ is a path of $B_5$ covering all vertices except $x_{j-1}$. By Case 2, $B_5$ is separable and has a cut vertex $z$ on $P_2[w,y_{n_2}]$. Because of the edge $wx_{j+1}$, $G-\{x_{j-1}, z\}$ has at most two components. Therefore $\{x_{j-1},z\}$ is not a $2$-cut and $zx_{j-1}$ is an edge. 

Finally note that $x_{j-1} \neq x$ and $w \neq y$ or else $B$ would contain a $3$-leaf spanning tree. Consider the $\{x_{j-1},w\}$-bridge $B_6$ containing $x$ and $y$. $T[V(B_6)]$ is connected so by Case 1, $B_6$ is separable and has a cut vertex $z'$. Since $|V(B_6)| \geq 4$, at least one of $\{x_{j-1},z'\}$ and $\{w,z'\}$ is a $2$-cut of $G$. However because of the edges $zx_{j-1}$ and $wx_{j+1}$, both $G - \{x_{j-1}, z'\}$ and $G - \{w, z'\}$ have at most two components. This gives the final contradiction.
\end{proof}
\end{section}

\begin{section}{Acknowledgements}
The author would like to thank Martin Merker and Carsten Thomassen for helpful discussions. This research was supported by ERC Advanced Grant GRACOL, project number 320812.
\end{section}

\bibliographystyle{plain}
 \bibliography{literature}

\end{document}